\documentclass[12pt,leqno]{amsart}
\usepackage[all]{xy}
\usepackage{amssymb,amsmath,amsthm,mathrsfs}
\usepackage{cases}
\usepackage{enumerate}
\usepackage{hyperref}
\usepackage{bm}
\usepackage{dsfont}

\numberwithin{equation}{section}
\allowdisplaybreaks[3]
\makeatother
\newtheorem{thm}{Theorem}[section]
\newtheorem{Con}[thm]{Conjecture}

\newtheorem{lem}[thm]{Lemma}

\theoremstyle{definition}

\numberwithin{equation}{section}

\theoremstyle{definition}
\newtheorem*{ack}{Acknowledgements}

\newcommand{\eps}{\varepsilon}

\newcommand{\si}{\sigma}

\newcommand{\re}{\textup{Re}}
\newcommand{\im}{\textup{Im}}

\newcommand{\abs}[1]{\left\lvert#1\right\rvert}

\begin{document}
\title[Moments of quadratic Dirichlet character sums]{Moments of quadratic Dirichlet character sums} %%%%%%%%%%%%
\author[Y. Toma]{Yuichiro Toma}
%\date{\today}
\address{Graduate School of Mathematics, Nagoya University, Chikusa-ku, Nagoya 464-8602, Japan.}
\email{m20034y@math.nagoya-u.ac.jp}

\makeatletter
\@namedef{subjclassname@2020}{\textup{2020} Mathematics Subject Classification}
\makeatother
\maketitle

\begin{abstract}
We consider moments of higher powers of quadratic Dirichlet character sums. In a restricted region, we give their asymptotic behavior by using de la Bret\`{e}che's multivariable Tauberian theorem. We also give the lower bound of the exponent of $\log$ factor in the conjecture of Jutila. As an application, we give a lower bound of a weighted average of shifted moments of quadratic Dirichlet $L$-functions.
\end{abstract} %%%%%%%%%

\section{Introduction}
Dirichlet character sums have been a central topic in analytic number theory and relate to many number theoretical
problems. Let $q \geq 2$ and $\chi$ be a non-principal Dirichlet character modulo $q$. Let
\[
M_q(\chi):=\max_{Y \leq q}\abs{\sum_{n \leq Y}\chi(n)}.
\]
In 1918, P\'{o}lya and Vinogradov independently obtained the best-known upper bound
\begin{align*}
    M_q(\chi) \ll \sqrt{q}\log q.
\end{align*}
Under the Generalized Riemann Hypothesis (GRH), Montgomery and Vaughan~\cite{MV77} improved the above bound to 
\[
M_q(\chi) \ll \sqrt{q}\log\log q.
\]
This bound is best possible since Paley~\cite{P} showed that there exists an infinite family of primitive quadratic characters $\chi$ mod $q$ such that
\[
M_q(\chi) \gg \sqrt{q}\log\log q.
\]
According to the parity of order of characters, the bounds of character sums are studied by several mathematicians, for example Granville and Soundararajan~\cite{GS07}, Goldmakher and Lamzouri~\cite{GL12, GL14}, Goldmakher~\cite{Gold12}, Lamzouri~\cite{L17} and Lamzouri and Mangerel~\cite{LM22}.

\subsection{Moments of quadratic character sums}
We consider moments of quadratic character sums. Let $S(X)$ denote the set of all non-principal quadratic Dirichlet characters of modulus at most $X$. In~\cite{Ju2}, Jutila made the following conjecture concerning higher powers of character sums.

\begin{Con}[Jutila]
\label{Jutila conjecture}
For all $k = 1, 2,\dots$ and $X \geq 3, Y \geq 1$, the estimate
    \begin{align}
\label{S_k}
    S_k(X,Y) := \sum_{\chi \in S(X)} \abs{\sum_{n \leq Y}\chi(n)}^{2k} \leq c_1(k)XY^k (\log X)^{c_2(k)}
\end{align}
holds with certain coefficients $c_1(k), c_2(k)$ depending on only $k$.
\end{Con}
The case $k = 1$ was first studied by Jutila~\cite{Ju1}, and he showed that $S_1(X,Y) \ll XY(\log X)^8$.  Armon~\cite{Ar} refined Jutila's result to $S_1(X,Y) \ll XY\log X$. In \cite{Vi02}, Virtanen established a weaker version of Conjecture \ref{Jutila conjecture} for the case $k = 2$ that $S_2(X,Y) \ll X^{1+\eps}Y^2$ for any $\eps>0$. 

For more general cases $k \geq 0$ to be a real number, Gao and Zhao~\cite{GZ23} considered a smoothed version. They treated the following primitive character sum and showed under the GRH that for large $X, Y$ and any real $k \geq 0$
\begin{align}
\sum_{\substack{0<d\leq X \\ d:\text{odd, squarefree}}}\abs{\sum_{n} \chi_{8d}(n)W\left(\frac{n}{Y}\right)}^{2k} &\ll \begin{cases}
    XY^k(\log X)^k & 0 \leq k<1, \\
    XY^k(\log X)^{9k-8} & 1 \leq k <2, \\
    XY^k(\log X)^{k(2k+1)} & k \geq 2 \\
\end{cases}
\end{align}
holds, where $W$ is any non-negative, smooth function compactly supported on the set of positive real numbers. Gao and Zhao also showed that the above estimate for case $0 \leq k \leq 1$ is unconditionally valid. 

On the other hand, we consider $S_k(X,Y)$ when $Y$ is `small' compared to $X$. In this situation, it becomes much simpler to estimate. However, it is still interesting because we can derive the lower bound of the exponent $c_2(k)$ of $\log X$ in (\ref{S_k}) without assuming the GRH. Hereafter, we assume that $k \in \mathbb{N}$. We prove the following.

\begin{thm}
\label{thm:1}
    %There exists $\delta>0$ such that f
    For fixed small $\eps>0$ and for large $X,Y$ with $X^{\eps\frac{1}{3k}} \ll Y \ll \left(X/(\log X)^2\right)^{\frac{1}{3k}}$, we have
    \[
    S_k(X,Y) \asymp_{k,\eps} XY^{k}(\log X)^{2k^2-k}.
    \]
\end{thm}
In the region $X^{\eps\frac{1}{3k}} \ll Y \ll \left(X/(\log X)^2\right)^{\frac{1}{3k}}$, Theorem \ref{thm:1} implies that the exponent $2k^2-k$ of $\log X$ is best possible. Therefore, it implies that if Conjecture \ref{Jutila conjecture} is true, then
\[
c_2(k) \geq 2k^2-k
\]
holds. The optimality of $c_2(k)$ has not been studied. Our result is the first example of the value of $c_2(k)$ and provides the sharp bounds for $c_2(k)$ in the special case. Recently, Gao and Zhao~\cite{GZ24} derived the shifted moments of quadratic Dirichlet $L$-functions, and as an application, they showed assuming the truth of the GRH that
\[
\sum_{\substack{0<d\leq X \\ d:\text{odd, squarefree}}}\abs{\sum_{n \leq Y} \chi_{8d}(n)}^{2k} \ll XY^k (\log X)^{2k^2-k+1}
\]
for real $k \geq (\sqrt{5}+1)/2$ and for large $X, Y$. By the same argument as in in the proof of Theorem \ref{thm:1}, we can deduce that $\sum_{\substack{0<d\leq X \\ d:\text{odd, squarefree}}} \abs{\sum_{n \leq Y} \chi_{8d}(n)}^{2k}$ has the same order as $S_k(X,Y)$ for $X^{\eps\frac{1}{3k}} \ll Y \ll \left(X/(\log X)^2\right)^{\frac{1}{3k}}$. Hence, from the above results, under the GRH
\[
XY^{k}(\log X)^{2k^2-k} \ll \sum_{\substack{0<d\leq X \\ d:\text{odd, squarefree}}} \abs{\sum_{n \leq Y} \chi_{8d}(n)}^{2k} \ll XY^{k}(\log X)^{2k^2-k+1}
\]
holds for large $X, Y$ with $Y \ll \left(X/(\log X)^2\right)^{\frac{1}{3k}}$.

Since $\log Y \asymp \log X$ when $X^{\eps\frac{1}{3k}} \ll Y \ll \left(X/(\log X)^2\right)^{\frac{1}{3k}}$, Theorem \ref{thm:1} follows immediately from the next theorem. 
\begin{thm}
\label{asymptotic for S_k(X,Y)}
    %There exists $\delta>0$ such that f
    For large $X,Y$ with $Y \ll \left(X/(\log X)^2\right)^{\frac{1}{3k}}$, we have
    \begin{align*}
    S_k(X,Y) &= XY^{k}Q(\log Y)\\
    &\quad +O(XY^{k-\theta} )+O( X^{\frac{1}{2}+\eps} Y^{k(2+\eps)})+O( X^{\frac{1}{2}}(\log X)Y^{\frac{5k}{2}}(\log Y)^k),
    \end{align*}
    where $Q \in \mathbb{R}[x]$ is a polynomial of exact degree $2k^2-k$ and $\theta>0$.
\end{thm}

\subsection{Shifted moments of quadratic Dirichlet $L$-functions}
The original work of Jutila on the estimations of quadratic character sums were applied to the second moment of quadratic Dirichlet $L$-functions on the critical line. In fact, by using the estimation $S_1(X,Y)\ll XY(\log X)^8$ and Gallagher's large sieve estimates, Jutila~\cite{Ju3} himself obtained the following upper bound 
\begin{align*}
    \sum_{\chi \in S(X)} \int_{-T}^T \abs{L(1/2+it,\chi)}^2 dt \ll XT (\log (X(T+1)))^{16}
\end{align*}
for $X \geq 3$ and $T>0$. Jutila~\cite{Ju4} also applied his estimation on $S_1(X,Y)$ to obtain the zero density estimates for quadratic Dirichlet $L$-functions. 

In this article, we apply Theorem \ref{thm:1} to the shifted moments of quadratic Dirichlet $L$-functions of the form
\begin{align*}
    M_{\bm{t}}(X):=\sum_{\substack{0<d\leq X \\ d:\text{odd, squarefree}}}\abs{L(1/2+it_1,\chi_{8d}) \cdots L(1/2+it_{2k},\chi_{8d})},
\end{align*}
where $\bm{t}=(t_1,\dots,t_{2k})$. This is a generalization in the sense of $t_j$'s shift of the moments of quadratic Dirichlet $L$-functions at the central point. It is conjectured to be 
    \begin{align*}
        \sum_{\substack{0<d\leq X \\ d:\text{odd, squarefree}}} L(1/2,\chi_{8d})^{2k} \sim C(k) X(\log X)^{2k^2+k}.
    \end{align*}
Jutila~\cite{Ju5} first dealt with this problem and computed the first and second moments, and Soundararajan~\cite{So} obtained asymptotic formulas for the second and third moments. The shifted moments $M_{\bm{t}}(X)$ is a quadratic analogue of shifted moments of the Dirichlet $L$-functions 
\begin{align}
    \label{Shifted moments q-aspect}
    \sum_{\substack{\chi \bmod q \\ \chi\text{:primitive}}} \abs{L(1/2+it_1,\chi)\cdots L(1/2+it_{2k},\chi)},
\end{align}
which was first introduced by Munsch~\cite{Munsch17}. Munsch gave its upper bound under the GRH. 

Recently, Gao and Zhao~\cite{GZ24} showed as a special case that under the GRH and $2k$-tuple $(t_1, \dots, t_{2k})$ with $\abs{t_j} \leq X^A$ for fixed positive real number $A$
\begin{align*}
    M_{\bm{t}}(X) &\ll X(\log X)^\frac{k}{2} \prod_{1 \leq j<l \leq 2k}\abs{\zeta(1+\abs{t_j-t_l}+1/\log X)}^\frac{1}{2}\abs{\zeta(1+\abs{t_j+t_l}+1/\log X)}^\frac{1}{2} \\
    & \quad \times \prod_{1 \leq j \leq 2k}\abs{\zeta(1+\abs{2t_j}+1/\log X)}^\frac{3}{4}.
\end{align*}
The above upper bound can further be estimated by
\begin{align*}
    M_{\bm{t}}(X) &\ll X(\log X)^\frac{k}{2} \prod_{1 \leq j<l \leq 2k}g(\abs{t_j-t_l})^\frac{1}{2}g(\abs{t_j+t_l})^\frac{1}{2} \prod_{1 \leq j \leq 2k}g(\abs{2t_j})^\frac{3}{4},
\end{align*}
where $g : \mathbb{R}_{\geq 0} \to \mathbb{R}$ is defined by
\begin{align*}
    g(x) = \begin{cases}
        \log X & \text{ if } x \leq 1/\log X \text{ or } x \geq e^X, \\
        1/x & \text{ if }  1/\log X \leq x \leq 10, \\
        \log\log x & \text{ if } 10 \leq x \leq e^X.
        \end{cases}
\end{align*}

The original problem of the shifted moments goes back to the case of the Riemann zeta-function. In~\cite{Ch11}, Chandee considered shifted moments of the Riemann zeta-function including 
\begin{align*}
    M_{\bm{\alpha}}(T) &:= \int_T^{2T} \abs{\zeta(1/2+i(t_1+\alpha_1))\cdots \zeta(1/2+i(t_{2k}+\alpha_{2k})} dt,
\end{align*}
and obtained sharp lower and upper bounds of $M_{\bm{\alpha}}(T)$ under the Riemann Hypothesis (RH). Later, Curran~\cite{Cu24, Cu24+} showed as a special case that for all $\abs{\alpha_j} \leq \frac{T}{2}$
\begin{align*}
    M_{\bm{\alpha}}(T) & \asymp T (\log T)^\frac{k}{2} \prod_{1 \leq j < l \leq 2k} \abs{\zeta(1+\abs{t_j-t_l}+1/\log T)}^\frac{1}{2}
\end{align*}
holds under the RH. 

In this article, let $F(X)$ and $G_X(t_1,\dots,t_{2k})$ be positive valued functions. We assume that  
    \begin{align}
    \label{shifted moment}
    M_{\bm{t}}(X) \asymp F(X) G_X(t_1,\dots,t_{2k})
    \end{align}
holds for $\bm{t}=(t_1,\dots,t_{2k})$ with $\abs{t_j} \leq X^\Delta$ for all $1 \leq j \leq 2k$ and a positive $\Delta \leq 1$. We focus on $G_X(t_1,\dots,t_{2k})$ because we only know the upper bounds of the order of $M_{\bm{t}}(X)$. The precise order of $G_X(t_1,\dots,t_{2k})$ is difficult to determine even assuming the GRH and is not easy to conjecture for general shifts. We prove a lower bound for a weighted average of $G_X(t_1,\dots,t_{2k})$. 

\begin{thm}
\label{thm:3}
    We assume that (\ref{shifted moment}) and the GRH holds. With the notation as above, we have
    \begin{align*}
    \int_{-X^\Delta}^{X^\Delta} \dots \int_{-X^\Delta}^{X^\Delta} \frac{G_X(t_1,\dots,t_{2k})}{\prod_{j=1}^{2k}(\abs{t_j}+1)} dt_1\cdots dt_{2k} &\gg \frac{X(\log X)^{2k^2-k}}{F(X)}.
    \end{align*}
\end{thm}

We combine our Theorem \ref{thm:3} with Lemma 4.2 in~\cite{GZ24} to deduce that under the GRH
    \begin{align*}
    \frac{X(\log X)^{2k^2-k}}{F(X)} \ll \int_{-X^\Delta}^{X^\Delta} \dots \int_{-X^\Delta}^{X^\Delta} \frac{G_X(t_1,\dots,t_{2k})}{\prod_{j=1}^{2k}(\abs{t_j}+1)} dt_1\cdots dt_{2k} \ll \frac{X(\log X)^{2k^2-k+1}}{F(X)}
    \end{align*}
holds for $\bm{t} \in [-X^\eps, X^\eps]^{2k}$. This implies that the exact order of the weighted average of $G_X(t_1,\dots,t_{2k})$ with the exception of a single $\log X$ can be obtained. 

\section{Quadratic character sums}
Every quadratic nonprincipal character $\chi_D(n)$ modulo $\abs{D}$ can be represented by the Kronecker symbol $\left(\frac{D}{n}\right)$ for $n>0$, where $D$ is not a square and $D \equiv 0,1 \pmod 4$. Every such $D$ can be uniquely expressed as $dm^2$, where $d$ is a fundamental discriminant. Then $d$ is either 
\begin{itemize}
    \item $d$ is squarefree, $d \equiv 1 \pmod 4$,
    \item $d=4N$, $N$ is squarefree, $N \equiv 2 \pmod 4$ or
    \item $d=4N$, $N$ is squarefree, $N \equiv 3 \pmod 4$.
\end{itemize}
From the definition of (\ref{S_k}) we have
\begin{align*}
     S_k(X,Y) &= \sideset{}{^*}{\sum}_{\abs{D} \leq X} \abs{\sum_{n \leq Y}\chi_D(n)}^{2k} \\
     &= \sum_{m \leq X^{\frac{1}{2}}} \sideset{}{^\flat}{\sum}_{\abs{d} \leq \frac{X}{m^2}} \abs{\sum_{\substack{n \leq Y \\ (n,m)=1}}\left(\frac{d}{n}\right)}^{2k} \\
     &= \sum_{m \leq X^{\frac{1}{2}}} \sideset{}{^\flat}{\sum}_{\abs{d} \leq \frac{X}{m^2}} \sum_{\substack{n_1 \leq Y \\ (n_1,m)=1}} \dots \sum_{\substack{n_{2k} \leq Y \\ (n_{2k},m)=1}}\left(\frac{d}{n_1}\right) \cdots \left(\frac{d}{n_{2k}}\right) \\
     &= \sum_{n_1 \leq Y} \dots \sum_{n_{2k} \leq Y } \sum_{\substack{m \leq X^{\frac{1}{2}} \\ (m,n_1\cdots n_{2k})=1}} \sideset{}{^\flat}{\sum}_{\abs{d} \leq \frac{X}{m^2}}\left(\frac{d}{n_1\cdots n_{2k}}\right),
\end{align*}
where $\sum^*$ and $\sum^\flat$ indicate that the sum is over not a square and $\equiv 0,1 \pmod 4$ and fundamental discriminants, respectively. Then we apply the following estimates of quadratic characters given by Granville and Soundararajan~\cite{GS}.

\begin{lem}{\cite[Lemma 4.1 and page 1014]{GS}}
\label{lem:GranvilleSoundararajan}
    Let $n$ be a positive integer, not a perfect square, then
    \[
    \sideset{}{^\flat}{\sum}_{\abs{d} \leq z} \left(\frac{d}{n}\right) \ll z^{\frac{1}{2}}n^{\frac{1}{4}}\log (2n),
    \]
    while if $n$ is a perfect square then
    \[
     \sideset{}{^\flat}{\sum}_{\abs{d} \leq z} \left(\frac{d}{n}\right) =\frac{z}{\zeta(2)} \prod_{p\mid n} \frac{p}{p+1}+O(z^{\frac{1}{2}+\eps} d(\sqrt{n})),
    \]
    where $d(n)=\sum_{d \mid n}1$.
\end{lem}
By using the above estimates, we obtain
\begin{align*}
    S_k(X,Y) &= \sum_{n_1 \leq Y} \dots \sum_{n_{2k} \leq Y } \sum_{\substack{m \leq X^{\frac{1}{2}} \\ (m,n_1\cdots n_{2k})=1}} \left\{\frac{X}{\zeta(2)m^2 }\prod_{\substack{p \mid n_1\cdots n_{2k} \\ n_1 \cdots n_{2k}=\square}} \frac{p}{p+1} \right. \\
    & \qquad\qquad +O \left( \left(\frac{X}{m^2}\right)^{\frac{1}{2}+\eps}(n_1\cdots n_{2k})^\frac{\eps}{2}\right) \\
    & \qquad\qquad \left.+O \left( \left( \frac{X}{m^2}\right)^\frac{1}{2}(n_1\cdots n_{2k})^\frac{1}{4}(\log (n_1\cdots n_{2k}))^{\frac{1}{2}}\right)\right\} \\
    &=\sum_{n_1 \leq Y} \dots \sum_{n_{2k} \leq Y}  \left\{\frac{X}{\zeta(2)}\prod_{\substack{p \mid n_1\cdots n_{2k} \\ n_1 \cdots n_{2k}=\square}} \frac{p}{p+1} \left( \zeta(2)\prod_{p \mid n_1 \cdots n_{2k}}\left( 1-\frac{1}{p^2}\right)+O(X^{-\frac{1}{2}})\right)\right. \\
    & \qquad\qquad +O \left( X^{\frac{1}{2}+\eps}(n_1\cdots n_{2k})^\frac{\eps}{2}\right) \\
    & \qquad\qquad \left.+O \left( X^\frac{1}{2}(\log X)(n_1\cdots n_{2k})^\frac{1}{4}(\log (n_1\cdots n_{2k}))^{\frac{1}{2}}\right)\right\} \\
    &=X\sum_{n_1 \leq Y} \dots \sum_{n_{2k} \leq Y} \prod_{\substack{p \mid n_1\cdots n_{2k} \\ n_1 \cdots n_{2k}=\square}} \left(1-\frac{1}{p}\right) +O(X^{\frac{1}{2}}Y^{2k}) \\
    & \qquad\qquad +O \left( X^{\frac{1}{2}+\eps}Y^{k(2+\eps)}\right) +O \left( X^\frac{1}{2}(\log X)Y^\frac{5k}{2}(\log Y)^k\right).
\end{align*}

\section{Arithmetical functions in several variables}
For a positive integer $r$, let $f : \mathbb{N}^r \to \mathbb{C}$ be an arithmetic function of $r$ variables. Then $f$ which is not identically zero is said to be multiplicative if $f(1,\dots,1)=1$ and
\[
f(m_1n_1,\dots,m_kn_r) = f(m_1,\dots,m_r)f(n_1,\dots,n_r)
\]
holds for any $m_1,\dots,m_r, n_1,\dots,n_r \in \mathbb{N}$ such that $\gcd(m_1 \cdots m_r, n_1 \cdots n_r) = 1$. This is a generalization of the classical one-variable multiplicative functions which satisfies $f(1)=1$ and $f(mn) = f(m)f(n)$ for $\gcd(m,n)=1$. 

If $f$ is multiplicative, then it is determined by the values $f(p^{v_1}, \dots, p^{v_r})$, where $p$ is prime and $v_1, \dots,v_r \in \mathbb{N}_0$. Thus, a formal Dirichlet series in several variables of $f$ admits an Euler product expansion:
\begin{align}
\label{Euler prod}
\sum_{n_1, \dots n_r =1}^\infty \frac{f(n_1,\dots, n_r)}{n_1^{s_1}\cdots n_r^{s_r}} =\prod_{p} \left( \sum_{v_1,\dots,v_r = 0}^\infty \frac{f(p^{v_1}, \dots, p^{v_r})}{p^{v_1s_1+\dots+v_ks_r}}\right).
\end{align}

In the literature on multiplicative functions in several variables, the corresponding multiple Dirichlet series first appears in the paper of Vaidyanathaswamy~\cite{V31}. The theory of multiple Dirichlet series was further developed by some authors without mentioning~\cite{V31}. In 1977, Selberg formulated multiple Dirichlet series attached to multiplicative functions in several variables (see~\cite[Chapter I, Definition 4.17]{T}). For the details of the theory of multiple Dirichlet series, we refer~\cite{To}.

In this paper, let
\begin{align}
\label{f r-variable}
    f(n_1,\dots,n_{2k}) := \prod_{p \mid n_1\cdots n_{2k}} \left(1-\frac{1}{p} \right)\times \mathds{1}_{n_1\cdots n_{2k} = \square}.
\end{align}
Here, $n=\square$ means that $n$ is a square integer, and $\mathds{1}_{n = \square}$ takes $1$ if $n=\square$, and $0$ otherwise. It is easy to show that $f(n_1,\dots,n_{2k})$ is multiplicative. From the previous section we find that
\begin{align*}
    S_k(X,Y) &= X\sum_{n_1 \leq Y} \dots \sum_{n_{2k} \leq Y} f(n_1,\dots,n_{2k}) +O(X^{\frac{1}{2}}Y^{2k}) \\
    & \qquad\qquad +O \left( X^{\frac{1}{2}+\eps}Y^{k(2+\eps)}\right) +O \left( X^\frac{1}{2}(\log X)Y^\frac{5k}{2}(\log Y)^k\right).
\end{align*}

\begin{lem}
\label{lem:convolution}
    We have   
    \begin{align}
        \sum_{n_1, \dots, n_{2k} =1}^\infty \frac{f(n_1,\dots, n_{2k})}{n_1^{s_1}\cdots n_k^{s_{2k}}} &= \left(\prod_{1 \leq j \leq 2k} \zeta(2s_j) \prod_{1 \leq l_1<l_2 \leq 2k} \zeta(s_{l_1}+s_{l_2})\right) E(s_1,\dots,s_{2k}), 
    \end{align}
    where $E(s_1,\dots,s_{2k})$ is a Dirichlet series absolutely convergent for $\re(s_j) > 1/4$ for $1\leq j \leq 2k$.
\end{lem}
\begin{proof}
Form (\ref{Euler prod}) and (\ref{f r-variable}), we have 
\begin{align*}
    \sum_{n_1, \dots n_{2k} =1}^\infty \frac{f(n_1,\dots, n_{2k})}{n_1^{s_1}\cdots n_{2k}^{s_{2k}}} 
    &= \prod_{p} \left( 1+\sum_{\substack{v_1,\dots,v_{2k} =0 \\ (v_1,\dots,v_{2k})\neq (0,\dots,0) \\ v_1+\dots+v_{2k}:\text{even}}}^\infty \left(1-\frac{1}{p}\right) \frac{1}{p^{v_1s_1+\dots+v_{2k}s_{2k}}} \right) \\
    &= \prod_{p} \left( 1+\sum_{\substack{v_1,\dots,v_{2k} =0 \\ v_1+\dots+v_{2k}=2}} \left(1-\frac{1}{p}\right) \frac{1}{p^{v_1s_1+\dots+v_{2k}s_{2k}}} \right. \\
    &\qquad\qquad \left. +\sum_{\substack{v_1,\dots,v_{2k} =0 \\ v_1+\dots+v_{2k}:\text{even} \\ v_1+\dots+v_{2k} \geq 4  }}^\infty \left(1-\frac{1}{p}\right) \frac{1}{p^{v_1s_1+\dots+v_{2k}s_{2k}}} \right).
    % &=\left(\prod_{1 \leq j \leq 2k} \zeta(2s_j) \prod_{1 \leq l_1<l_2 \leq 2k} \zeta(s_{j_1}+s_{j_2})\right) E(s_1,\dots,s_{2k}), 
\end{align*}
By putting $\si = \min \{ \si_j \mid 1 \leq j \leq 2k\}$ we have
\begin{align*}
    \sum_{\substack{v_1,\dots,v_{2k} =0 \\ v_1+\dots+v_{2k}:\text{even} \\ v_1+\dots+v_{2k} \geq 4  }}^\infty \left(1-\frac{1}{p}\right) \frac{1}{p^{v_1s_1+\dots+v_{2k}s_{2k}}} &\ll \sum_{q=2}^\infty \frac{1}{p^{2q\si}} \sum_{\substack{v_1,\dots,v_{2k}\geq 0 \\ v_1+\dots+v_{2k}=2q}}1 \\
    &\ll_k \sum_{q=2}^\infty \frac{(2q+1)^{2k}}{p^{2q\si}}.
\end{align*}
We now invoke the generating function for the sum of $l$-th powers
\[
\sum_{l=1}^\infty l^n x^l = \frac{x}{(1-x)^{n+1}}A_n(x),
\]
where $A_n(x)$ are the Eulerian polynomials which are defined by the exponential generating function
\[
\sum_{n=0}^\infty A_n(t) \frac{x^n}{n!} = \frac{t-1}{t-e^{(t-1)x}}.
\]
By using the facts that $A_n(x)$ can be written in terms of the Eulerian numbers $A(n,k)$ as
\begin{align*}
    A_n(x) = \sum_{k=0}^n A(n,k) x^k
\end{align*}
and $A(n,0)=1$ for all $n \geq 1$, we get
\begin{align*}
    \sum_{\substack{v_1,\dots,v_{2k} =0 \\ v_1+\dots+v_{2k}:\text{even} \\ v_1+\dots+v_{2k} \geq 4  }}^\infty \left(1-\frac{1}{p}\right) \frac{1}{p^{v_1s_1+\dots+v_{2k}s_{2k}}} 
    \ll_k \sum_{q=2}^\infty \frac{q^{2k}}{p^{2q\si}} \ll_k \frac{1}{p^{4\si}}.
\end{align*}

Next we factor out the products of the Riemann zeta-functions to obtain
\begin{align}
\label{multiple sum in Euler product}
    \sum_{n_1, \dots n_{2k} =1}^\infty \frac{f(n_1,\dots, n_{2k})}{n_1^{s_1}\cdots n_{2k}^{s_{2k}}} 
    &= \left(\prod_{1 \leq j \leq 2k} \zeta(2s_j) \prod_{1 \leq l_1<l_2 \leq 2k} \zeta(s_{l_1}+s_{l_2})\right) E(s_1,\dots,s_{2k}), 
\end{align}
say. Then by using the estimate (\ref{multiple sum in Euler product}), we can calculate $E(s_1,\dots,s_{2k})$ as
\begin{align*}
    &E(s_1,\dots,s_k) \\
    &= \prod_{p} \left( 1+\sum_{\substack{v_1,\dots,v_{2k} \geq 0 \\ v_1+\dots+v_{2k}=2}} \left(1-\frac{1}{p}\right) \frac{1}{p^{v_1s_1+\dots+v_{2k}s_{2k}}}  +\sum_{\substack{v_1,\dots,v_{2k} =0 \\ v_1+\dots+v_{2k}:\text{even} \\ v_1+\dots+v_{2k} \geq 4  }}^\infty \left(1-\frac{1}{p}\right) \frac{1}{p^{v_1s_1+\dots+v_{2k}s_{2k}}} \right) \\
    &\qquad \times \prod_{1 \leq j \leq 2k} \left( 1-\frac{1}{p^{2s_j}}\right) \prod_{1 \leq l_1<l_2 \leq 2k} \left( 1-\frac{1}{p^{s_{l_1}+s_{l_2}}}\right) \\
    &=\prod_{p} \left( 1+\sum_{\substack{1 \leq i \leq 2k \\ \\ \\}} \left(1-\frac{1}{p}\right) \frac{1}{p^{2s_i}} +\sum_{1 \leq h_1<h_2 \leq 2k} \left(1-\frac{1}{p}\right) \frac{1}{p^{s_{h_1}+s_{h_2}}} +O_k\left(\frac{1}{p^{4\si}}\right) \right) \\
    &\qquad \times \left(1 - \sum_{1 \leq j \leq 2k} \frac{1}{p^{2s_j}}+ O_k\left(\frac{1}{p^{4\si}}\right)\right) \left(1-\sum_{1 \leq l_1<l_2 \leq 2k} \frac{1}{p^{s_{l_1}+s_{l_2}}}+O_k\left( \frac{1}{p^{4\si}}\right)\right) \\
    &=\prod_{p} \left( 1-2\sum_{\substack{1 \leq i \leq 2k \\ \\ \\}} \frac{1}{p^{2s_i+1}} -2\sum_{1 \leq h_1<h_2 \leq 2k} \frac{1}{p^{s_{h_1}+s_{h_2}+1}}+O_k \left(\frac{1}{p^{4\si}}\right) \right).    
\end{align*}
Therefore, the above Euler product is convergent absolutely for $\re(s_j)>1/4$. So we obtain the desired result.
\end{proof}

\section{Proof of Theorem \ref{asymptotic for S_k(X,Y)}}
Our main tool to prove Theorem \ref{asymptotic for S_k(X,Y)} is the de la Bret\`{e}che Tauberian theorem~\cite{B} for non-negative arithmetical functions in several variables. We denote the canonical basis of $\mathbb{C}^r$ by $\{\bm{e}_j\}_{j=1}^r$ and its dual basis by $\{\bm{e}_j^*\}_{j=1}^r$.

\begin{thm}[{\cite[Th\'{e}or\`{e}me 1]{B}}]
\label{thm:breteche1}
    Let $f: \mathbb{N}^r \to \mathbb{R}$ be a non-negative function and $F$ the associated Dirichlet series of $
f$ defined by
\[
F(\bm{s}) = F(s_1,\dots,s_r) = \sum_{n_1,\dots,n_r=1}^\infty \frac{f(n_1,\dots,n_r)}{n^{s_1}\cdots n_r^{s_r}}.
\]
Denote by $\mathcal{LR}_r^+(\mathbb{C})$ the set of non-negative $\mathbb{C}$-linear forms from $\mathbb{C}^r$ to $\mathbb{C}$ on $(\mathbb{R}_{\geq 0})^r$. Moreover, assume that there exists $(c_1, \dots,c_r) \in (\mathbb{R}_{\geq 0})^r$ such that:
\begin{enumerate}[(1)]
\item For $\bm{s} \in \mathbb{C}^r$, $F(s_1,\dots,s_r)$ is absolutely convergent for $\re(s_i) > c_i$ for all $1 \leq i \leq r$.

\item There exist a finite family $\mathcal{L} = (l^{(i)})_{1\leq i \leq q}$ of non-zero elements of $\mathcal{LR}_r^+(\mathbb{C})$, a finite family $(h^{
(i)})_{1\leq i \leq q^\prime}$ of elements of $\mathcal{LR}_r^+(\mathbb{C})$ and $\delta_1, \delta_3>0$ such that the function $H$ defined by
\[
H(\bm{s}) = F(\bm{s} + \bm{c}) \prod_{i=1}^q l^{(i)}(\bm{s})
\]
has a holomorphic continuation to the domain
\begin{align*}
D(\delta_1, \delta_3) 
&= \left\{ \bm{s} \in \mathbb{C}^r \mid \re\left(l^{(i)}(\bm{s})\right) > -\delta_1 \text{ for all } 1\leq i \leq q \right. \\
& \qquad\qquad\qquad \left. \text{ and } \re\left(h^{(i)}(\bm{s})\right)> -\delta_3 \text{ for all } 1\leq i \leq q^\prime \right\}.
\end{align*}
\item There exists $\delta_2>0$ such that for any $\eps,\eps^\prime>0$ we have uniformly in $\bm{s} \in D(\delta_1 -\eps,\delta_3-\eps^\prime)$ 
\[
H(\bm{s}) \ll \prod_{i=1}^q \left(\abs{\im\left(l^{(i)}(\bm{s})\right)}+1 \right)^{1-\delta_2 \min \left\{0, \re\left( l^{(i)}(\bm{s})\right)\right\}}(1+(\abs{\im(s_1)}+\dots+\abs{\im(s_k)})^\eps).
\]
\end{enumerate}
Set $J = J(\mathbb{C}) = \{j \in \{1, \dots, r\} \mid c_j = 0\}$. Denote $w$ to be the cardinality of $J$ and by $j_1 <\dots<j_w$ its elements in increasing order. Define the $w$ linear forms $l^{(q+i)}$ $(1 \leq i \leq w)$ by $l^{(q+i)}(\bm{s}) = {e^{*}}_{j_i}(\bm{s}) = s_{j_i}$.

Then, for any $\bm{\beta} = (\beta_1,\dots, \beta_r) \in (0,\infty)^r$, there exist a polynomial $Q_{\bm{\beta}} \in \mathbb{R}[X]$ of degree at most $q + w - Rank \left(l^{(1)}, \dots, l^{(q+w)}\right)$ and $\theta > 0$ such that as $x \to \infty$
\[
\sum_{n_1\leq x^{\beta_1}}\dots\sum_{n_r \leq x^{\beta_r}} f(n_1,\dots,n_r) = x^{\langle \bm{c}, \bm{\beta} \rangle} Q_{\bm{\beta}}(\log x) +O\left( x^{\langle \bm{c}, \bm{\beta} \rangle - \theta}\right).
\]
Here, $\langle \cdot, \cdot \rangle$ denotes the usual inner product in $\mathbb{R}^r$.
\end{thm}

The next theorem gives a determination of the precise degree of the polynomial $Q_{\bm{\beta}}$ appearing in the previous theorem. We denote by $\mathbb{R}_*^+$ the set of strictly positive real numbers, and the notation $con^*(\{l^{(1)},\dots,l^{(q)}\})$ means $\mathbb{R}_*^+ l^{(1)}+\dots+\mathbb{R}_*^+ l^{(q)}$.

\begin{thm}[{\cite[(iv) of Th\'{e}or\`{e}me 2]{B}}]
\label{thm:breteche2}
Let $f : \mathbb{N}^r \to \mathbb{R}$ be a non-negative function satisfying the assumptions of Theorem \ref{thm:breteche1}. Let $\bm{\beta} = (\beta_1,\dots,\beta_r) \in (0,\infty)^r$. Set $\mathcal{B} = \sum_{i=1}^r\beta_i \bm{e}_i^* \in \mathcal{LR}_r^+(\mathbb{C})$. If $Rank\left(l^{(1)},\dots,l^{(q+w)} \right)=n, H(\bm{0})\neq 0$ and $\mathcal{B} \in con^* \left(\{ l^{(1)},\dots,l^{(q+w)} \} \right)$, then $\operatorname{deg}(Q_{\bm{\beta}}) =  q + w - n$.
\end{thm}

\begin{proof}[Proof of Theorem \ref{asymptotic for S_k(X,Y)}]
From (\ref{f r-variable}), we can easily find that $f(n_1,\dots,n_{2k})$ is non-negative and in Lemma \ref{lem:convolution}, we proved that $f(n_1, \dots,n_k)$ has an absolutely convergent series $F(s)$ for $\re(s_j) > 1/2$ for all $1 \leq j \leq 2k$. This shows that $f(n_1, \dots,n_{2k})$ satisfies (1) of Theorem \ref{thm:breteche1}. 

Next, we write $\bm{1/2} = (1/2,\dots,1/2)$. Then $F(\bm{s} + \bm{1/2})$ is an absolutely convergent series for $\re(s_j)> 0$. Therefore, we define the function
\begin{align}
\label{H(s)}
H(\bm{s}):=F(\bm{s} + \bm{1/2})\left(\prod_{1 \leq j \leq 2k}s_j \right)\left( \prod_{1 \leq l_1 <l_2 \leq 2k}(s_{l_1}+s_{l_2}) \right).    
\end{align}
So, we take $\mathcal{L} = \{s_j \mid 1 \leq j \leq 2k\} \cup \{s_{l_1}+s_{l_2} \mid 1 \leq l_1<l_2 \leq 2k \}$ in (2) of Theorem \ref{thm:breteche1}.
Then, by Lemma \ref{lem:convolution} it can be rewritten as
\begin{align*}
    H(\bm{s})&=\left(\prod_{1 \leq j \leq 2k}\zeta(2s_j+1)s_j \right)\left( \prod_{1 \leq l_1 <l_2 \leq 2k}\zeta(s_{l_1}+s_{l_2}+1)(s_{l_1}+s_{l_2}) \right)E(\bm{s}+\bm{1/2}).
\end{align*}
For $j \in \{1,\dots,2k\}$, there exists $\delta_1 \in (0,1/4)$ such that $\zeta(2s_j+1)s_j$ has analytic continuation to the plane $\re(s_j)>-\delta_1$. Similarly, $\zeta(s_{l_1}+s_{l_2}+1)(s_{l_1}+s_{l_2})$ also has analytic continuation to the plane $\re(s_j)>-\delta_1$. 

Furthermore, $E(\bm{s}+\bm{1/2})$ is holomorphic in $\re(s_j)>-\delta_1$ for $1 \leq j \leq 2k$. Since $c_j = 1/2$ for all $1 \leq j \leq 2k$, we can take $h^{(j)}(\bm{s}) = s_j$ in the notation of Theorem \ref{thm:breteche1}. Then we put $\delta_3=\delta_1$. Therefore $f(n_1,\dots,n_{2k})$ also satisfies (2) of Theorem \ref{thm:breteche1}.

Finally, for $\re(s_j) > -1/4$, by applying the convexity bound of the Riemann zeta-function,
\begin{align*}
\zeta(s_j+1)s_j &\ll (\abs{s_j}+1)^{1- \frac{1}{2}\min \{0, \re(s_j)\}}, \\
\zeta(s_{l_1}+s_{l_2}+1)s_j &\ll (\abs{s_{l_1}+s_{l_2}}+1)^{1- \frac{1}{2}\min \{0, \re(s_{l_1}+s_{l_2})\}}
\end{align*}
hold. The above argument shows that $H(s)$ satisfies (3) of Theorem \ref{thm:breteche1} with $\delta_2 = 1/2$.

Since $q= k+\binom{k}{2} = k +k(2k-1), w=0$ and the rank of the linear forms in $\mathcal{L}$ is $k$, we have 
\[
\sum_{n_1,\dots,n_k \leq Y} f(n_1,\dots,n_k) = Y^\frac{k}{2} Q(\log Y) + O(Y^{\frac{k}{2}-\theta}),
\]
where $Q(\log Y)$ is a polynomial of degree at most $2k^2-k$.

The remained task is to determine the degree of the polynomial $Q$ by using Theorem \ref{thm:breteche2}. Since $\bm{c}=\bm{1/2}$, we know that $w=0$, and then $Rank\left(l^{(1)},\dots,l^{(q)} \right) =k$. Moreover as $s_j, s_{l_1}+s_{l_2} \to 0$, we have
\begin{align*}
    \zeta(2s_j+1)s_j \to 1/2, \quad \zeta(s_{l_1}+s_{l_2}+1)(s_{l_1}+s_{l_2}) \to 1.
\end{align*}
From the Euler product, it holds that $E(\bm{1/2})$ does not vanish. Hence $H(\bm{0})\neq 0$. At last, we see that $\mathcal{B} = \sum_{i=1}^{2k} \bm{e}_i^*(\bm{s}) \in con^* \left(\{ l^{(1)},\dots,l^{(q)} \} \right)$ for $\bm{e}_i^*(\bm{s}) = s_i$. Therefore by Theorem \ref{thm:breteche2} (2), we have $\operatorname{deg}(Q) = 2k^2-k$. Therefore, we complete the proof.
\end{proof}

\section{Proof of Theorem \ref{thm:3}}
In order to prove Theorem \ref{thm:3}, we follow the same strategy as Gao and Zhao~\cite{GZ24}. For a function $f(x)$, we recall that the Mellin transform $\widehat{f}(s)$ with $s \in \mathbb{C}$ is defined by
\begin{align*}
    \widehat{f}(s)= \int_0^\infty f(x)x^{s-1} dx.
\end{align*}

Let $\Psi_U$ be a smooth, non-negative function such that $\Psi_U(x) \leq 1$ for all $x$ and that $\Psi_U$ is supported on $(0.1)$ and that $\Psi_U$ satisfies
$\Psi(x) = 1$ for $x \in (1/U,1-1/U)$. 

\begin{lem}
\label{lem:GZ}
    Assume the GRH. With the notation as above, we have for $k \geq 1$,
    \begin{align*}
        \sum_{\substack{0< d \leq X \\ d \text{:odd, squarefree}}} \abs{\chi_{8d}(n)\left(1-\Psi_{X^{\frac{\Delta}{2}}}\left(\frac{n}{Y}\right)\right)}^{2k} \ll XY^k.
    \end{align*}
\end{lem}
\begin{proof}
    This is Lemma 4.3 in~\cite{GZ24}.
\end{proof}

\begin{lem}
\label{lem:GZ2}
    Assume that the GRH. With the notation as above, we have for $k \geq 1$,
    \begin{align*}
        &\sum_{\substack{0< d \leq X \\ d \text{:odd, squarefree}}} \abs{\sum_{n=1}^\infty \chi_{8d}(n)\Psi_{X^{\frac{\Delta}{2}}}\left(\frac{n}{Y}\right)}^{2k} \\
        &\quad \ll Y^k \sum_{\substack{0< d \leq X \\ d \text{:odd, squarefree}}} \abs{\int_{-X^\Delta}^{X^\Delta} \abs{L(1/2+it,\chi_{8d})}\frac{dt}{1+\abs{t}}}^{2k}+O\left( XY^k\right).
    \end{align*}
\end{lem}
\begin{proof}
    This has been shown in the proof of Theorem 1.3 in~\cite{GZ24}.
\end{proof}

Now we are ready to prove Theorem \ref{thm:3}. Let $X$ be a large number. By the same argument as in Theorem \ref{thm:1}, we have
\begin{align}
\label{S_8k lower bound}
    \sum_{\substack{0< d \leq X \\ d \text{:odd, squarefree}}} \abs{\sum_{n \leq Y}\chi_{8d}(n)}^{2k} \gg XY^k(\log X)^{2k^2-k}
\end{align}
for $Y$ satisfying $X^{\eps \frac{1}{3k}} \ll Y \ll \left(X/(\log X)\right)^{\frac{1}{3k}}$. 
By inserting  the function $\Psi_{X^{\frac{\Delta}{2}}}\left(\frac{n}{Y}\right)$, we get
\begin{align}
\begin{split}
\label{S_8k upper bound}
    & \sum_{\substack{0< d \leq X \\ d \text{:odd, squarefree}}} \abs{\sum_{n \leq Y}\chi_{8d}(n)}^{2k} \\
    & \quad \ll  \sum_{\substack{0< d \leq X \\ d \text{:odd, squarefree}}} \abs{\sum_{n=1}^\infty \chi_{8d}(n) \Psi_{X^{\frac{\Delta}{2}}}\left(\frac{n}{Y}\right)}^{2k} + \sum_{\substack{0< d \leq X \\ d \text{:odd, squarefree}}} \abs{\sum_{n \leq Y}
    \chi_{8d}(n) \left(1-\Psi_{X^{\frac{\Delta}{2}}}\left(\frac{n}{Y}\right)\right)}^{2k}.
\end{split}
\end{align}
By combining (\ref{S_8k lower bound}), (\ref{S_8k upper bound}), Lemma \ref{lem:GZ} and Lemma \ref{lem:GZ2}, we obtain
\begin{align*}
    \sum_{\substack{0< d \leq X \\ d \text{:odd, squarefree}}} \abs{\int_{-X^\Delta}^{X^\Delta} \abs{L(1/2+it,\chi_{8d})}\frac{dt}{1+\abs{t}}}^{2k} \gg X(\log X)^{2k^2-k}.
\end{align*}
We expand the left hand side in the above and use the assumption to obtain
\begin{align*}
    & \sum_{\substack{0< d \leq X \\ d \text{:odd, squarefree}}} \abs{\int_{-X^\Delta}^{X^\Delta} \abs{L(1/2+it,\chi_{8d})}\frac{dt}{1+\abs{t}}}^{2k} \\
    & =\int_{-X^\Delta}^{X^\Delta} \dots \int_{-X^\Delta}^{X^\Delta} \sum_{\substack{0<d\leq X \\ d:\text{odd, squarefree}}}\abs{L(1/2+it_1,\chi_{8d}) \dots L(1/2+it_{2k},\chi_{8d})}\frac{dt_1\cdots dt_{2k}}{\prod_{j=1}^{2k}(\abs{t_j}+1)} \\
    &\asymp F(X) \int_{-X^\Delta}^{X^\Delta} \dots \int_{-X^\Delta}^{X^\Delta} \frac{G_X(t_1,\dots,t_{2k})}{\prod_{j=1}^{2k}(\abs{t_j}+1)} dt_1\cdots dt_{2k}.
\end{align*}
Hence under the GRH we have 
\begin{align*}
\int_{-X^\Delta}^{X^\Delta} \dots \int_{-X^\Delta}^{X^\Delta} \frac{G_X(t_1,\dots,t_{2k})}{\prod_{j=1}^{2k}(\abs{t_j}+1)} dt_1\cdots dt_{2k} \gg \frac{X(\log X)^{2k^2-k}}{F(X)}.
\end{align*}
So we complete the proof of Theorem \ref{thm:3}.

\begin{ack} 
The author would like to thank Dr. Dilip Kumar Sahoo for informing the author of multiple Dirichlet series of Selberg's sense. The author would also like to thank Professor Kohji Matsumoto for his valuable comments. Some parts of this research were carried out during his stay at IISER Berhampur in India from June to July in 2024. The author would like to thank Professor Kasi Viswanadham Gopajosyula and the staff of IISER Berhampur for their support and hospitality. The author is supported by Grant-in-Aid for JSPS Research Fellow (Grant Number:24KJ1235).
\end{ack}

\end{document}